\documentclass[a4paper, 12pt]{article}
\usepackage{fullpage}
\usepackage{amsmath, amsthm, amssymb, mathrsfs, graphicx, subfigure}
\usepackage{ifthen, url}
\usepackage[usenames]{color}

\usepackage[sort&compress]{natbib}
\bibpunct{(}{)}{;}{a}{}{,}

\theoremstyle{plain}
\newtheorem{thm}{Theorem}

\newtheorem{lem}{Lemma}
\newtheorem{prop}{Proposition}

\theoremstyle{remark}

\theoremstyle{definition}
\newtheorem{defn}{Definition}

\theoremstyle{remark}

\theoremstyle{definition}
\newtheorem*{condS}{\it Condition S}
\newtheorem*{condP}{\it Condition P}

\newcommand{\FF}{\mathbb{F}}
\newcommand{\KK}{\mathbb{K}}
\newcommand{\NN}{\mathbb{N}}

\newcommand{\YY}{\mathbb{Y}}

\newcommand{\F}{\mathscr{F}}
\newcommand{\Y}{\mathscr{Y}}

\newcommand{\E}{\mathsf{E}}

\newcommand{\prob}{\mathsf{P}}

\newcommand{\eps}{\varepsilon}
\renewcommand{\phi}{\varphi}

\newcommand{\Bigmid}{\; \Bigl\vert \;}

\newcommand{\ftrue}{{f^\star}}
\newcommand{\fbest}{{f^\circ}}
\newcommand{\fhat}{\hat f}
\newcommand{\fbar}{\bar f}

\newcommand{\nm}{\mathsf{N}}

\newcommand{\Kbar}{\bar K}
\newcommand{\Vbar}{\bar V}

\begin{document}

\title{On convergence rates of Bayesian predictive densities and posterior distributions}
\author{
Ryan Martin \\ 
Department of Mathematics, Statistics, and Computer Science \\ 
University of Illinois at Chicago \\ 
\url{rgmartin@math.uic.edu} \\
\mbox{} \\
Liang Hong \\
Department of Mathematics \\
Bradley University \\
\url{lhong@bradley.edu}
}
\date{\today}

\maketitle

\begin{abstract}
Frequentist-style large-sample properties of Bayesian posterior distributions, such as consistency and convergence rates, are important considerations in nonparametric problems.  In this paper we give an analysis of Bayesian asymptotics based primarily on predictive densities.  Our analysis is unified in the sense that essentially the same approach can be taken to develop convergence rate results in iid, mis-specified iid, independent non-iid, and dependent data cases.  
\medskip

\end{abstract}

\section{Introduction}
\label{S:intro}

In Bayesian nonparametric problems, asymptotic concentration properties of the posterior distribution are often key to motivating a particular choice of prior.  Indeed, for infinite-dimensional problems, elicitation of subjective priors is difficult and a theory of objective priors for remains elusive, so large-sample properties of the posterior are often what drives the choice of prior.  A desirable (frequentist) property can be summarized roughly as follows: for a given ``true model'' and prior, as more and more data becomes available, the posterior distribution becomes more and more concentrated around this true model with large probability.  Early efforts along these lines are given in \citet{doob1949} and \citet{schwartz1965}.  Stronger results, some including rates of convergence, are presented in \citet{bsw1999}, \citet{ggr1999}, \citet{ggv2000}, \citet{shen.wasserman.2001}, \citet{ghosalvaart2001, ghosalvaart2007}, \citet{tokdar2006}, and \citet{walker2007}.  Modern efforts include extensions to non-Euclidean sample spaces \citep{bhattacharya.dunson.2010}, more complex models \citep{pati.dunson.tokdar.2011}, and misspecified models \citep{kleijn, shalizi2009, lian2009}. 

This paper presents a sort of unified analysis of Bayesian posterior convergence rates based on predictive densities.  These predictive densities are fundamental quantities in Bayesian statistical inference---they are the Bayes estimates of the density under a variety of loss functions.  This connection with predictive densities is not new, but the extent to which we depend on these quantities gives our analysis a strong Bayesian flavor.  Moreover, we show how essentially the same techniques can be used to develop convergence rate theorems for a variety of models, including iid, mis-specified iid, independent non-iid, and dependent data.  In particular, we prove (apparently) new Cesaro-style convergence rate results for predictive densities, in each of the four contexts above, under weaker conditions than usual for posterior convergence rate theorems.  We also develop a fundamental lemma, also based on predictive densities, which helps bound the numerator of the posterior probability for sets not-too-close to the true data-generating density.  This result is similar to Proposition~1 in \citet{walker2007}, but the proof is different and applies almost word-for-word in a variety of contexts.  It also relies an interesting notion of separation of points and sets, apparently due to \citet{choi.ramamoorthi.2008}.  This lemma is then used to prove posterior convergence rate theorems.  

The remainder of the paper is organized as follows.  Section~\ref{S:prelim} develops the notation and terminology used throughout the paper, in particular, the notion of prior thickness at the true data-generating density, and separation of sets from this same density.  Sections~\ref{S:predictive} and \ref{S:posterior} cover predictive density and posterior convergence rates, respectively, both in the simplest iid context.  In Section~\ref{S:posterior}, we prove an auxiliary result that demonstrates the sieve+covering style conditions in \citet{ggv2000} are weaker than the prior summability conditions in \citet{walker2007}.  The results on predictive density and posterior convergence rate theorems are extended to the mis-specified iid, independent non-iid, and dependent cases in Sections~\ref{S:miss}--\ref{S:markov}, respectively.  Finally, some concluding remarks are given in Section~\ref{S:discuss}.

\section{Bayesian nonparametrics}
\label{S:prelim}

\subsection{Notation and definitions}  
\label{SS:notation}

Let $\YY$ be a Polish space equipped with its Borel $\sigma$-algebra $\Y$.  Suppose $Y_{1:n} = (Y_1,\ldots,Y_n)$, $n \geq 1$, are independent $\YY$-valued observations with common distribution $F$, and that $F$ has a density $f = dF/d\mu$ with respect to some $\sigma$-finite measure $\mu$ on $(\YY,\Y)$.  Let $\FF$ be the set of all such densities $f$ and $\F$ its Borel $\sigma$-algebra.  Then a prior distribution $\Pi$ for $f$ is a probability measure on the measurable space $(\FF,\F)$.  Following Bayes' theorem, the posterior distribution of $f$, given $Y_{1:n}$, can be written as  
\begin{equation}
\label{eq:post1}
\Pi_n(A) = \Pi(A \mid Y_{1:n}) = \frac{\int_A \prod_{i=1}^n f(Y_i) \, \Pi(df)}{\int_{\FF} \prod_{i=1}^n f(Y_i) \, \Pi(df)}, \quad A \in \F.
\end{equation}
We require a topology on $\FF$, and here we focus on the Hellinger topology.  The Hellinger distance $H$ is given by $H(f,f') = \{\int (f^{1/2}-{f'}^{1/2})^2\,d\mu\}^{1/2}$, for $f,f' \in \FF$.  

For describing large-sample properties of the posterior, it is standard to assume that there is a ``true density'' $\ftrue$ from which the data $Y_1,\ldots,Y_n$ are observed.  It shall be required that the prior $\Pi$ puts a sufficient amount of mass around this $\ftrue$; the precise conditions on $\Pi$ are stated in Section~\ref{SS:thickness}.  With ``true density'' $\ftrue$, it is typical to rewrite the posterior \eqref{eq:post1} as 
\begin{equation}
\label{eq:post2}
\Pi_n(A) = \frac{\int_A R_n(f) \,\Pi(df)}{\int_{\FF} R_n(f) \,\Pi(df)}, \quad A \in \F, 
\end{equation}
where $R_0(f) \equiv 1$ and 
\begin{equation}
\label{eq:lr}
R_n(f) = \prod_{i=1}^n f(Y_i) / \ftrue(Y_i), \quad n \geq 1.
\end{equation}
In what follows, we will occasionally refer to the posterior $\Pi_n$, restricted to a given set $A$.  By that we mean the measure $\Pi_n^A$ defined as $\Pi_n^A(\cdot) = \Pi_n(A \cap \cdot) / \Pi_n(A)$.  Also, $\lesssim$ and $\gtrsim$ will denote inequality up to a universal constant.

Convergence rates of the posterior $\Pi_n$ concerns the amount of probability assigned to (expanding) sets that do not contain the true density $\ftrue$ when $n$ is large.  Let $(\eps_n)$ be a positive vanishing sequence.  Then the posterior $\Pi_n$ has a Hellinger convergence rate $\eps_n$ if $\Pi_n(\{f: H(\ftrue,f) \gtrsim \eps_n\}) \to 0$ in probability.  Here, and in what follows, the ``in probability'' qualification is with respect to $\prob=\prob_{\ftrue}^\infty$, the product distribution, under $\ftrue$, of the infinite data sequence $Y_{1:\infty}=(Y_1,Y_2,\ldots)$.

\subsection{Prior support conditions}
\label{SS:thickness}

In order for the posterior distribution to concentrate around $\ftrue$, some support conditions on the prior $\Pi$ are needed.  For example, if there exists a set $A \ni \ftrue$ such that $\Pi(A) = 0$, then, trivially, the posterior cannot concentrate around $\ftrue$.  To avoid these kinds of degeneracies, it is typical to assume that $\Pi$ satisfies the Kullback--Leibler property, i.e., that $\Pi(\{f: K(\ftrue,f) < \eps\}) > 0$ for all $\eps > 0$, where $K(\ftrue,f) = \int \log(\ftrue/f)\ftrue\,d\mu$ is the Kullback--Leibler divergence of $f$ from $\ftrue$.   See \citet{wu.ghosal.2008, wu.ghosal.2010} for sufficient conditions and a host of examples.  While the Kullback--Leibler property itself is not a necessary condition for posterior convergence, it does make up part of an important and useful set of sufficient conditions, developed by \citet{schwartz1965}.  Indeed, the Kullback--Leibler property alone is enough to imply that the posterior is weakly consistent.  But extra conditions, beyond the Kullback--Leibler property, are generally needed to establish strong consistency; see \citet{choi.ramamoorthi.2008}.  

To establish rates of convergence, something even stronger than the usual Kullback--Leibler property is needed.  Set $V(\ftrue,f) = \int \{\log(\ftrue/f)\}^2 \ftrue \,d\mu$.  

\begin{defn}
\label{def:thick}
Let $(\eps_n)$ be a positive sequence such that $\eps_n \to 0$ and $n\eps_n^2 \to \infty$.  The prior $\Pi$ is said to be $\eps_n$-\emph{thick} at $\ftrue$ if, for some constant $C > 0$, 
\begin{equation}
\label{eq:thick}
\Pi(\{f: K(\ftrue,f) \leq \eps_n^2,\, V(\ftrue,f) \leq \eps_n^2\}) \geq e^{-Cn\eps_n^2}.
\end{equation}
\end{defn}

This is exactly condition (2.4) in \citet{ggv2000}, which they motivate with entropy considerations.  Since \eqref{eq:thick} is stronger than the Kullback--Leibler property, prior thickness can be seen as a support condition on the prior, guaranteeing that the prior assigns a sufficient amount of mass near $\ftrue$.  Beyond this intuition, the following technical lemma, giving a lower bound on the denominator in \eqref{eq:post2}, is a consequence of prior thickness.  See \citet[][Lemma~8.1 and the proof of Theorem~2.1]{ggv2000}.  

\begin{lem}
\label{lem:thick}
Let $I_n = \int R_n(f) \,\Pi(df)$ be the denominator in \eqref{eq:post2}.  If $\Pi$ is $\eps_n$-thick at $\ftrue$, then $\prob(I_n \leq e^{-cn\eps_n^2}) \to 0$ for any $c > C+1$ with $C$ as in \eqref{eq:thick}.  
\end{lem}

Next is a simple application of Lemma~\ref{lem:thick}, similar to Proposition~4.4.2 in \citet{ghoshramamoorthi}, that will be used in the proof of the main results.  

\begin{lem}
\label{lem:prior.to.posterior}
Assume $\Pi$ is $\eps_n$-thick at $\ftrue$.  For a sequence $(U_n) \subset \F$, suppose that $\Pi(U_n) \lesssim e^{-rn\eps_n^2}$, where $r > C+1$, with $C$ as in \eqref{eq:thick}.  Then $\Pi_n(U_n) \to 0$ in probability.  
\end{lem}

\begin{proof}
Write $\Pi_n(U_n) = L_n/I_n$.  Using Markov's inequality and the assumption on $\Pi(U_n)$, it is easy to check that $\prob(L_n > e^{-cn\eps_n^2}) \lesssim e^{-(r - c)n\eps_n^2}$.  Therefore, if $c \in (C+1, r)$, then $\prob(e^{cn\eps_n^2}L_n > \eta) \to 0$ for any $\eta > 0$.  Similarly, from Lemma~\ref{lem:thick}, for the same $c$, $\prob(I_n \leq e^{-cn\eps_n^2}) \to 0$.  Then by the law of total probability, 
\begin{equation}
\label{eq:total.probability}
\begin{split}
\prob\{\Pi_n(U_n) > \eta\} & = \prob(L_n/I_n > \eta) \\
& = \prob(L_n/I_n > \eta, I_n \leq e^{-cn\eps_n^2}) + \prob(L_n/I_n > \eta, I_n > e^{-cn\eps_n^2}) \\
& \leq \prob(I_n \leq e^{-cn\eps_n^2}) + \prob(L_n/I_n > \eta, I_n > e^{-cn\eps_n^2}) \\
& \leq \prob(I_n \leq e^{-cn\eps_n^2}) + \prob(e^{cn\eps_n^2}L_n > \eta).
\end{split}
\end{equation} 
Both quantities on the right-hand side vanish with $n$, so $\Pi_n(U_n) \to 0$ in probability.  
\end{proof}

\subsection{Convexity and separation}
\label{SS:convexity}

\citet{choi.ramamoorthi.2008} make use of two important properties for subsets $A$ of $\FF$.  Here we define and discuss these properties.  

\begin{defn}
\label{def:convex}
A set $A \subseteq \FF$ is convex if, for any probability measure $\Phi$ supported on $A$, the expectation, $f_\Phi = \int f \,\Phi(df)$, also belongs to $A$.  
\end{defn}

Examples of convex subsets of $\FF$ include balls, i.e., all those $f$ within a specified distance from a center $f_0$.  For an important example, let $h=H^2/2$, a slight modification of the squared Hellinger distance.  Choose a point $f_0 \in \FF$ and let $A=\{f: h(f_0,f) \leq r\}$.  Now, take any probability measure $\Phi$ supported on $A$.  Then by convexity of $h$ and definition of $A$, we have 
\[ h(f_0,f_\Phi) \leq \int_A h(f_0,f) \,\Phi(df) \leq r. \]
Therefore, $f_\Phi$ is in $A$ and, hence, $A$ is convex.  In the applications that follow, the probability measure $\Phi$ will often be a truncated version of the posterior distribution.  

\begin{defn}
\label{def:separate}
 A density $\ftrue \in \FF$ and a set $A \subseteq \FF$ are $\delta$-separated (with respect to $h$) if $h(\ftrue,f) > \delta$ for all $f$ in $A$.  
\end{defn}

For an important example, choose $r > 0$ and $f_0$ such that $H(\ftrue,f_0) > r$.  Then $\ftrue$ and $A=\{f: H(f_0,f) < r/2\}$ are $\delta$-separated, with $\delta=r^2/8$.  To see this, note that the triangle inequality implies
\[ H(\ftrue,f) \geq H(\ftrue,f_0) - H(f_0,f) \quad \forall\; f \in A. \]
From the definitions of $f_0$ and $A$, the right-hand side is strictly greater than $r/2$.  Therefore, $h(\ftrue,f) > r^2/8$, so $\ftrue$ and $A$ are $(r^2/8)$-separated (with respect to $h$).  

In the applications that follow, we shall extend this idea in two directions.  First, in some cases, we need separation with respect to distances other than Hellinger distance $H$ (or $h$).  Second, we shall consider sequences of sets $(A_n)$ and sequences of numbers $(\delta_n)$.  Then the notion of $\delta_n$-separation of a density $\ftrue$ and sets $A_n$ is straightforward.

\section{Convergence rates for predictive densities}
\label{S:predictive}

Predictive densities are fundamental quantities in Bayesian analysis.  Indeed, they are the Bayes density estimators under a variety of different loss functions.  In particular, the predictive density of $Y_i$, given $Y_1,\ldots,Y_{i-1}$, is 
\[ \fhat_{i-1}(y) = \int f(y) \Pi_{i-1}(df), \]
the posterior expectation of $f(y)$.  For example, in a density estimation problem with Hellinger distance as the loss function, the predictive density $\fhat_n$ is the Bayes estimator of $f$ in the sense that it minimizes Bayes risk.  

Our first result develops a Kullback--Leibler convergence rate for predictive densities in a Cesaro sense.  The proof is based on calculations in \citet{barron1987}.  

\begin{prop}
\label{prop:barron}
For a given vanishing sequence $(\eps_n)$, let $\KK_n = \{f: K(\ftrue,f) \leq \eps_n^2\}$.  If $\log\Pi(\KK_n) \gtrsim -n\eps_n^2$, then $n^{-1} \sum_{i=1}^n \E\{ K(\ftrue, \fhat_{i-1}) \} \lesssim \eps_n^2$.  
\end{prop}

\begin{proof}
Let ${\ftrue}^n$ denote the joint density for an iid sample $(Y_1,\ldots,Y_n)$, i.e., the $n$-fold product of the $\ftrue$.  Likewise, let $\fhat^n$ denote the joint density of $(Y_1,\ldots,Y_n)$ under the Bayesian model with prior $\Pi$, i.e., $\fhat^n = \int f^n \, \Pi(df)$.  Since densities are non-negative, 
\[ \int_\FF f^n \, \Pi(df) \geq \int_{\KK_n} f^n \, \Pi(df) = \Pi(\KK_n) \int f^n \, \Pi^{\KK_n}(df), \]
where $\Pi^{\KK_n}$ is the prior $\Pi$ restricted and normalized to $\KK_n$.  Therefore, if we define $\Pi(\KK_n) \fhat^{n,\KK_n}$ as the lower bound above, then 
\begin{align*}
n^{-1} K({\ftrue}^n, \fhat^n) & \leq n^{-1} \{ K({\ftrue}^n, \fhat^{n,\KK_n}) - \log\Pi(\KK_n)\} \\
& \leq n^{-1} \int_{\KK_n} K({\ftrue}^n, f^n) \, \Pi^{\KK_n}(df) - n^{-1} \log \Pi(\KK_n), 
\end{align*}
where the last inequality is by convexity of $K$.  Recall the chain rule for the Kullback--Leibler number between product densities: $K(\ftrue^n, f^n) = nK(\ftrue,f)$.  Therefore, 
\[ n^{-1} K({\ftrue}^n, \fhat^n) \leq \int_{\KK_n} K({\ftrue}, f) \, \Pi^{\KK_n}(df) - n^{-1} \log \Pi(\KK_n). \]
By definition of $\KK_n$, the first term in the upper bound above is $\leq \eps_n^2$, and, by the assumption on $\Pi(\KK_n)$, the second term is $\lesssim \eps_n^2$.  

To complete the proof, we must connect $n^{-1} K(\ftrue^n,\fhat^n)$ and the average in the statement of the proposition.  For this we show that $\fhat^n(Y_1,\ldots,Y_n)$ factors as a product $\prod_{i=1}^n \fhat_{i-1}(Y_i)$ of predictive densities.  The key is 
\begin{align*}
\fhat^n(Y_1,\ldots,Y_n) & = \int \prod_{i=1}^n f(Y_i) \, \Pi(df) \\
& = \int f(Y_n) \prod_{i=1}^{n-1} f(Y_i) \, \Pi(df) \\
& = \int f(Y_n) \, \Pi_{n-1}(df) \cdot \int \prod_{i=1}^{n-1} f(Y_i) \, \Pi(df). 
\end{align*}
The first term in the last line is the expectation of $f(Y_n)$ with respect to the posterior distribution $\Pi_{n-1}$, which is exactly $\fhat_{n-1}(Y_n)$; the second term is the normalizing constant for $\Pi_{n-1}$.  The next step is to apply the same trick to the normalizing constant.  That is, write it as an expectation of $f(Y_{n-1})$ with respect to the posterior distribution $\Pi_{n-2}$ times a new normalizing constant.  This gives 
\[ \fhat^n(Y_1,\ldots,Y_n) = \fhat_{n-1}(Y_n) \fhat_{n-2}(Y_{n-1}) \int \prod_{i=1}^{n-2} f(Y_i) \,\Pi(df). \]
Continuing like this, we find that $\fhat^n(Y_1,\ldots,Y_n)$ factors as $\prod_{i=1}^n \fhat_{i-1}(Y_i)$.  Now, 
\begin{align*}
K(\ftrue^n, \fhat^n) & = \E \log \bigl\{ \ftrue^n(Y_1,\ldots,Y_n)/\fhat^n(Y_1,\ldots,Y_n) \bigr\} \\
& = \sum_{i=1}^n \E \log\{\ftrue(Y_i)/\fhat_{i-1}(Y_i)\} \\
& = \sum_{i=1}^n \E \bigl[ \E \bigl\{ \log \bigl( \ftrue(Y_i) / \fhat_{i-1}(Y_i) \bigr) \mid \Y_{i-1} \bigr\} \bigr].
\end{align*}
The conditional expectation is $K(\ftrue, \fhat_{i-1})$, so $K(\ftrue^n, \fhat^n)$ equals $\sum_{i=1}^n \E\{K(\ftrue, \fhat_{i-1})\}$.  This, together with the $\lesssim \eps_n^2$ bound on $n^{-1}K(\ftrue^n,\fhat^n)$ completes the proof.  
\end{proof}

Observe that the assumption of Proposition~\ref{prop:barron} is implied by $\eps_n$-thickness of $\Pi$ at $\ftrue$.  Also, the Kullback--Leibler divergence can be replaced by the Hellinger distance via the well-known inequality $h \lesssim K$.  That is, $n^{-1} \sum_{i=1}^n \E\{h(\ftrue, \fhat_{i-1})\} \lesssim \eps_n^2$.  

For another perspective, let $\fbar_n = n^{-1} \sum_{i=1}^n \fhat_{i-1}$, an average of predictive densities.  By convexity of $h$, $h(\ftrue, \fbar_n) \leq n^{-1} \sum_{i=1}^n h(\ftrue, \fhat_{i-1})$.  Therefore, Proposition~\ref{prop:barron} says that, if the prior is suitably concentrated around $\ftrue$, then $H(\ftrue, \fbar_n) = O_P(\eps_n)$.  As \citet{walker2003} explains, the prior $\Pi$ would have to be rather strange for this not to imply convergence of the predictive density $\fhat_n$ itself at the same $\eps_n$ rate.  

It is interesting that the predictive densities, and averages thereof, are asymptotically well-behaved with only local properties of the prior \citep{barron1999}.  This is particularly important because posterior convergence rates involve a compromise between local and global global properties.  For example, overall posterior convergence rates are determined by $\max\{\eps_n,\eps_n'\}$, where $\eps_n'$ gives a global characterization of the complexity of the model.  In many cases, $\eps_n'$ is bigger than $\eps_n$, slowing down the overall posterior convergence rate; see \citet{ghosalvaart2001}.  Proposition~\ref{prop:barron} requires no global conditions, so there is nothing slowing down convergence.  So, although Proposition~\ref{prop:barron} is a weaker result than full posterior convergence, it does provide some nice intuition.

\section{Convergence rates for the posterior}
\label{S:posterior}

\subsection{Review of existing results}
\label{SS:main}

There are essentially two kinds of theorems: the first kind makes assumptions on the ``size'' of the model $\FF$, and the second kind makes assumptions on how the prior probabilities are spread across $\FF$.  Before proving the convergence rate theorem, we discuss these conditions in more detail.  In particular, we show in Proposition~\ref{prop:equivalent} that the latter assumption is stronger than the former.  Throughout this discussion, we silently assume that the prior $\Pi$ is $\eps_n$-thick at $\ftrue$, with constant $C$ given in \eqref{eq:thick}.  

The first set of sufficient conditions are like those in \citet{ggv2000}.  Their concern is the existence of a suitable high mass, low entropy sieve.  Let $(\FF_n)$ be an increasing sequence of measurable subsets of $\FF$.  The idea is that the sieve $\FF_n$ will be large enough to contain all the reasonable $f$'s, but also small enough to be covered by a relatively small number of Hellinger balls that are each easier to work with.  Let $N(\eps_n, \FF_n, H)$ denote the Hellinger $\eps_n$-covering number of $\FF_n$, that is, the minimum number of Hellinger balls of radius $\eps_n$ needed to cover $\FF_n$.  Theorem~2.1 of \citet{ggv2000} assumes that the following condition (``S'' for sieve) holds:

\begin{condS}
There exists a sieve $\FF_n \subset \FF$ such that, for sufficiently large $n$, 
\begin{itemize}
\item[(a)] $\Pi(\FF_n^c) \lesssim e^{-rn\eps_n^2}$, where $r > C+1$, and 
\vspace{-2mm}
\item[(b)] $\log N(\eps_n, \FF_n, H) \lesssim n\eps_n^2$.  
\end{itemize}
\end{condS}

Part (a) ensures that $\Pi$ assigns most of its probability to a large subset of $\FF$, and Part (b) guarantees that this ``large'' subset of $\FF$ is not too large.  As opposed to prior thickness, which is a local property, S(a) and S(b) are global properties.  These conditions have, along with prior thickness, been verified for a variety of important priors, including Dirichlet process mixtures. 

Despite the nice geometric intuition of Condition~S, identifying a suitable sieve is sometimes difficult in practice.  Fortunately, there is an alternative sufficient condition (``P'' for prior), due to \citet{walker2007}, which can be easier to work with.

\begin{condP}
Let $B_n = \{f: H(\ftrue,f) > \eps_n\}$.  For $(A_{n,j})_{j \geq 1}$ a covering of $B_n$ by Hellinger balls of radius $\delta_n < \eps_n$, and some constants $c > 0$ and $\beta > 1$, the following holds:
\[ e^{-cn\eps_n^2} \sum_{j \geq 1} \Pi(A_{n,j})^{1/\beta} \to 0. \]
\end{condP}

The case $\beta=2$ was considered in \citet{walker2007}.  This condition ensures that the prior is sufficiently concentrated near $\ftrue$.  That is, if the prior is too spread out, then those covering sets could get large enough posterior probability that the summation above is of exponential order.  An advantage of Condition~P is that it is directly related to the Bayesian problem, through the prior probabilities, and often these prior probabilities have a nice form.  And by allowing $\beta \approx 1$, the Condition~P here is weaker than that in \citet{walker2007} with $\beta=2$.  However, in applications, the sets $A_{n,j}$ are typically assigned exponentially small prior probability so it is not clear if $\beta \in (1,2)$ is easier to verify or leads to any improvement in the rate of convergence.  

We claim that Condition~S is, in a certain sense, more fundamental than Condition~P, despite the fact that the latter is often easier in practice.  To justify this claim, we prove that Condition~S is actually weaker than Condition~P.  This connection between the two sets of conditions is implicit in Theorem~5 of \citet{ghosalvaart2007}.  An analogous result is given in \citet[][Theorem~4.4]{choi.ramamoorthi.2008} in the context of posterior consistency.  

\begin{prop}
\label{prop:equivalent}
Condition~P implies Condition~S.    
\end{prop}

\begin{proof}
Without loss of generality, suppose that, for each $n$, the sets $A_{n,j}$ are ordered such that $\Pi(A_{n,1}) \geq \Pi(A_{n,2}) \geq \cdots$.  Also, let $S_n = \sum_j \Pi(A_{n,j})^{1/\beta}$, which can be expressed as $S_n = e^{cn\eps_n^2 - v(n)}$ for some $v(n) > 0$ such that $v(n) \to \infty$.  Take $r > C+1$, and set 
\[ \FF_n = \bigcup_{j=1}^{J_n} A_{n,j}, \quad \text{where} \quad J_n = \min\{j \in \NN: j^{\beta-1} \geq S_n^\beta e^{rn\eps_n^2}\}. \]
Clearly, $\log N(\eps_n, \FF_n, H) = \log J_n \leq (\frac{r+\beta c}{\beta-1})n\eps_n^2 \lesssim n\eps_n^2$, so $\FF_n$ satisfies Condition~S(b).  Next, the special ordering of $\Pi(A_{n,j})$ implies that $J \Pi(A_{n,J})^{1/\beta} \leq \sum_{j=1}^J \Pi(A_{n,j})^{1/\beta} \leq S_n$ for any $J$, which in turn implies that $\Pi(A_{n,J}) \leq S_n^\beta/J^\beta$.  Therefore, 
\[ \Pi(\FF_n^c) = \Pi\Bigl(\bigcup_{j > J_n} A_{n,j} \Bigr) \leq \sum_{j > J_n} \Pi(A_{n,j}) \leq \sum_{j > J_n} \frac{S_n^\beta}{j^\beta} \lesssim \frac{S_n^\beta}{J_n^{\beta-1}} \leq e^{-rn\eps_n^2}, \]
so Condition~S(a) holds as well.  
\end{proof}

\begin{thm}
\label{thm:rates}
Suppose $\Pi$ is $\eps_n$-thick at $\ftrue$.  If either Condition~S or Condition~P holds, then $\Pi_n(\{f: H(\ftrue, f) \gtrsim \eps_n\}) \to 0$ in probability.  
\end{thm}

\begin{proof}
The part involving Condition~P follows from Proposition~\ref{prop:equivalent} and the part involving Condition~S.  The part involving Condition~S is proved in Section~\ref{SS:proof1}.
\end{proof}

Although the result in Theorem~\ref{thm:rates} is known, the proof that follows will highlight the importance of predictive densities in the study of posterior convergence rates.  The basic idea is that sets $A_n$ in $\FF$ such that the predictive densities, restricted to $A_n$, are not too close to $\ftrue$ will have vanishing posterior probability.  These predictive densities are fundamental quantities in the Bayesian context, so the proof presented herein perhaps has a better Bayesian interpretation compared to those arguments based on, say, existence of a consistent sequence of tests.

\subsection{A preliminary result}

Recall that $\Pi_k$ denotes the posterior distribution of $f$, given $Y_1,\ldots,Y_k$.  For a set $A_n$, we write $\Pi_k^{A_n}$ for that same posterior distribution, but restricted and normalized to $A_n$.  Then we can define a corresponding predictive density:
\[ \fhat_{i-1}^{A_n}(Y_i) = \int_{A_n} f(Y_i) \,\Pi_{i-1}^{A_n}(df), \quad i=1,\ldots,n. \]
For a sequence of sets $(A_n)$, let $L_{n,i} = \int_{A_n} R_i(f) \,\Pi(df)$ be the numerator of $\Pi_i(A_n)$ in \eqref{eq:post2}, $i=1,\ldots,n$.  Note that $L_{n,0} = \Pi(A_n)$.  It is easy to check that 
\[ L_{n,i} / L_{n,i-1} = \fhat_{i-1}^{A_n}(Y_i) / \ftrue(Y_i), \quad i=1,\ldots,n. \]
For $\Y_{i-1}$ the $\sigma$-algebra generated by $Y_1,\ldots,Y_{i-1}$, it follows that 
\begin{equation}
\label{eq:conditional.expectation}
\E\{ (L_{n,i} / L_{n,i-1})^{1/2} \mid \Y_{i-1}\} = 1-h(\ftrue, \fhat_{i-1}^{A_n}). 
\end{equation}
The next result, akin to Proposition~1 in \citet{walker2007}, provides a convenient fixed-$n$ bound the expected value of $L_{n,n}^{1/2}$ for suitable sets $A_n$.  

\begin{lem}
\label{lem:rates}
Let $\Pi$ be $\eps_n$-thick at $\ftrue$, with $C$ as in \eqref{eq:thick}.  If $A_n$ is convex and $d\eps_n^2$-separated from $\ftrue$, with $d > C+1$, then $\E(L_{n,n}^{1/2}) \leq \Pi(A_n)^{1/2} e^{-dn\eps_n^2}$.   
\end{lem}

\begin{proof}
Start with the ``telescoping product''
\[ \Bigl( \frac{L_{n,n}}{\Pi(A_n)} \Bigr)^{1/2} = \Bigl( \frac{L_{n,n}}{L_{n,0}} \Bigr)^{1/2} = \prod_{i=1}^n \Bigl( \frac{L_{n,i}}{L_{n,i-1}} \Bigr)^{1/2}. \]
Taking expectation of both sides, conditioning on $\Y_{i-1}$ and using \eqref{eq:conditional.expectation}, gives 
\[ \frac{\E(L_{n,n}^{1/2})}{\Pi(A_n)^{1/2}}  = \E \Bigl[ \prod_{i=1}^n \E \Bigl\{\Bigl( \frac{L_{n,i}}{L_{n,i-1}} \Bigr)^{1/2} \Bigmid \Y_{i-1} \Bigr\} \Bigr] = \E\Bigl[ \prod_{i=1}^n \{1-h(\ftrue,\fhat_{i-1}^{A_n})\} \Bigr]. \]
The assumed convexity of $A_n$ and its separation from $\ftrue$ together imply that 
\[ \frac{\E(L_{n,n}^{1/2})}{\Pi(A_n)^{1/2}} \leq (1 - d\eps_n^2)^n \leq e^{-dn\eps_n^2}. \]
Multiplying both sides by $\Pi(A_n)^{1/2}$ completes the proof.  
\end{proof}

The thickness assumption in Lemma~\ref{lem:rates} is not necessary, but it helps to set the notation for its primary application.  This use of ratios of predictive densities is not new; see \citet{walker2004a} and \citet{walker2007}.  In fact, Lemma~\ref{lem:rates} is similar to the main conclusion in Proposition~1 of \citet{walker2007}, although the proof is a bit different.

\subsection{Proof of Theorem~\ref{thm:rates}}
\label{SS:proof1}

For $M$ a sufficiently large constant to be determined, define $B_n = \{f: H(\ftrue,f) > M\eps_n\}$.  For the given $\FF_n$, it is clear that $\Pi_n(B_n) \leq \Pi_n(\FF_n^c) + \Pi_n(B_n \cap \FF_n)$.  From Condition~S(a) and Lemma~\ref{lem:prior.to.posterior}, we conclude that $\Pi_n(\FF_n^c) \to 0$ in probability.  We now turn attention to the second term, $\Pi_n(B_n \cap \FF_n)$. 

Choose a covering $B_n \cap \FF_n \subseteq \bigcup_{j=1}^{J_n} A_{nj}$, where each $A_{nj}$ is a Hellinger ball of radius $M\eps_n/2$ with center in $B_n$.  By Condition~S(b), $J_n = e^{Rn\eps_n^2}$ for some $R > 0$.  Now, since probabilities are $\leq 1$, we have
\[ \Pi_n(B_n \cap \FF_n) \leq \sum_{j=1}^{J_n} \Pi_n(A_{nj}) \leq \sum_{j=1}^{J_n} \Pi_n(A_{nj})^{1/2} = \frac{1}{I_n^{1/2}} \sum_{j=1}^{J_n} L_{n,nj}^{1/2}, \]
where $I_n=\int R_n(f)\,\Pi(df)$ is the denominator of $\Pi_n(A_{nj})$, which is independent of $j$, and $L_{n,nj} = \int_{A_{nj}} R_n(f) \,\Pi(df)$ is the numerator.  From the triangle inequality argument in the example following Definition~\ref{def:separate}, we know that $\ftrue$ and $A_{nj}$ are $(M^2\eps_n^2/8)$-separated for all $j=1,\ldots,J_n$.  So, provided that $M$ is sufficiently large, we may apply Lemma~\ref{lem:rates} to bound the expectation of the sum of $L_{n,nj}^{1/2}$.  Indeed, 
\[ \sum_{j=1}^{J_n} \E(L_{n,nj}^{1/2}) \leq \sum_{j=1}^{J_n} \Pi(A_{nj})^{1/2} e^{-(M^2/8)n\eps_n^2} \leq J_n e^{-(M^2/8) n \eps_n^2} = e^{-(M^2/4-2R)n\eps_n^2/2}. \]
If we choose $M$ such that $M^2 > 4[(C+1) + 2R]$, then (application of Lemma~\ref{lem:rates} is valid and) the upper bound above vanishes as $n \to \infty$.  Now let $S_n = \sum_{j=1}^{J_n} L_{n,nj}^{1/2}$, and pick $c \in (C+1, M^2/4-2R)$.  By Markov's inequality we have 
\[ \prob(e^{cn\eps_n^2/2} S_n > \eta) \lesssim e^{-M^2/4-2R-c)n\eps_n^2/2} \to 0, \quad \forall \; \eta > 0. \]
Also, by Lemma~\ref{lem:thick}, we have $\prob(I_n \leq e^{-cn\eps_n^2}) \to 0$.  A total-probability argument like in the proof of Lemma~\ref{lem:prior.to.posterior} gives
\begin{align*}
\prob\{\Pi_n(B_n \cap \FF_n) > \eta\} & \leq \prob(S_n/I_n^{1/2} > \eta) \\
& = \prob(S_n/I_n^{1/2} > \eta, I_n \leq e^{-cn\eps_n^2}) + \prob(S_n/I_n^{1/2} > \eta, I_n > e^{-cn\eps_n^2}) \\
& \leq \prob(I_n \leq e^{-cn\eps_n^2}) + \prob(S_n/I_n^{1/2} > \eta, I_n > e^{-cn\eps_n^2}) \\
& \leq \prob(I_n \leq e^{-cn\eps_n^2}) + \prob(e^{cn\eps_n^2/2}S_n > \eta).
\end{align*}
Since both of these terms vanish, we conclude that $\Pi_n(B_n) \leq \Pi_n(\FF_n^c) + \Pi_n(B_n \cap \FF_n) \to 0$ in probability, i.e., $\Pi_n(\{f: H(\ftrue,f) > M\eps_n\}) \to 0$ in probability.

\section{Extension to mis-specified iid models}
\label{S:miss}

\subsection{Notation and setup}

It can happen that the true density $\ftrue$ lies outside the support of prior.  In such cases, the posterior cannot concentrate around $\ftrue$.  However, the posterior can exhibit concentration properties around a different point in $\FF$.  Specifically, take $\fbest$ to be the $f \in \FF$ that minimizes the Kullback--Leibler divergence, i.e., 
\begin{equation}
\label{eq:kl.min}
K^\star(\fbest, f) := K(\ftrue,f) - K(\ftrue, \fbest) \geq 0, \quad \forall \; f \in \FF.  
\end{equation}
An analysis of posterior concentration is presented in \citet{kleijn}.  They show that, under certain conditions, the posterior distribution concentrates around the point $\fbest$.  Indeed, if 
\[ R_n(f) = \prod_{i=1}^n f(Y_i)/\fbest(Y_i), \]
then the posterior is given by 
\begin{equation}
\label{eq:post.miss}
\Pi_n(A) = \Pi(A \mid Y_1,\ldots,Y_n) = \frac{\int_A R_n(f) \,\Pi(df)}{\int_\FF R_n(f) \, \Pi(df)}. 
\end{equation}
Then the goal is to show that $\Pi_n(B_n^c) \to 0$, where $B_n$ is a shrinking neighborhood of $\fbest$.  Here we give an analysis based primarily on predictive densities.  First, we recall/revise some of our previous notions.  

\begin{description}
\item[\it Prior thickness.] Let $V^\star(\fbest,f) = \int \{\log(\fbest/f)\}^2 \ftrue \,d\mu$.  Then we have the following analogue of Definition~\ref{def:thick}, i.e., the prior $\Pi$ is $\eps_n$-thick at $\fbest$ if, for some constant $C > 0$, 
\begin{equation}
\label{eq:thick.miss}
\Pi(\{f: K^\star(\fbest,f) \leq \eps_n^2, \, V^\star(\fbest,f) \leq \eps_n^2\}) \geq e^{-Cn\eps_n^2}. 
\end{equation}
It follows from Lemma~7.1 of \citet{kleijn} that the result of Lemma~\ref{lem:thick} above holds in the mis-specified case, i.e., 
\begin{equation}
\label{eq:thick.lemma.miss}
\text{$\prob(I_n \leq e^{-cn\eps_n^2}) \to 0$ for any $c > C+1$}, 
\end{equation}
where $I_n = \int R_n(f) \,\Pi(df)$ is the denominator in \eqref{eq:post.miss}.  


\item[\it Separation.] For a distance on $\FF$, consider a weighted Hellinger distance $H^\star$, whose square is given by $H^\star(f,f')^2 = \int (f^{1/2}-{f'}^{1/2})^2 (\ftrue/ \fbest) \,d\mu$.  In the well-specified case, i.e., $\fbest=\ftrue$, this is the usual Hellinger distance.  Since $\int (f/\fbest) \ftrue \,d\mu \leq 1$ for all $f \in \FF$ \citep[][Lemma~2.3]{kleijn}, we have 
\begin{align*}
H^\star(\fbest,f)^2 & = \int \Bigl\{ \Bigl(\frac{f}{\fbest} \Bigr)^{1/2}-1\Bigr\}^2 \ftrue \,d\mu \\
& = 1 + \int \frac{f}{\fbest} \ftrue \,d\mu - 2 \int \Bigl( \frac{f}{\fbest} \Bigr)^{1/2} \ftrue \,d\mu \\
& \leq 2 - 2 \int \Bigl( \frac{f}{\fbest} \Bigr)^{1/2} \ftrue \,d\mu.
\end{align*}
Write $h^\star(\fbest,f) = 1-\int (f/\fbest)^{1/2}\ftrue \,d\mu$, so that ${H^\star}^2/2 \leq h^\star$.  Now we say that $\fbest$ and a set $A$ are $\delta$-separated (with respect to $h^\star$) if $h^\star(\fbest,f) > \delta$ for all $f \in A$.  
\end{description}

\subsection{Convergence rate results}

First, we extend Proposition~\ref{prop:barron} to the mis-specified case.  The only noticeable change is the use of the Kullback--Leibler contrast $K^\star(\fbest,f)$ in \eqref{eq:kl.min} instead of $K(\ftrue,f)$.  

\begin{prop}
\label{prop:barron.miss}
For a sequence $(\eps_n)$, with $\eps_n \to 0$, let $\KK_n=\{f: K^\star(\fbest,f) \leq \eps_n^2\}$.  If $\log \Pi(\KK_n) \gtrsim -n\eps_n^2$, then $n^{-1} \sum_{i=1}^n \E\{K^\star(\fbest,\fhat_{i-1})\} \lesssim \eps_n^2$.  
\end{prop}

\begin{proof}
Similar to that of Proposition~\ref{prop:barron}; use convexity of $K^\star$.  
\end{proof}

As before, if $\fbar_n$ is the average predictive density, $\fbar_n = n^{-1}\sum_{i=1}^n \fhat_{i-1}$, then Proposition~\ref{prop:barron.miss} and convexity of the Kullback--Leibler contrast implies $K^\star(\fbest,\fbar_n) = O_P(\eps_n^2)$.  Also, the condition on $\Pi(\KK_n)$ is implied by prior thickness \eqref{eq:thick.miss}.  Therefore, just like in the well-specified case, with only a local thickness condition on the prior, the predictive densities, or averages thereof, converge to the ``best'' density $\fbest$ in the model $\FF$.  

Towards a posterior concentration result, given sets $(A_n)$ in $\FF$, let $L_{n,i}$ be the numerator of $\Pi_i(A_n)$ in \eqref{eq:post.miss}; note, $L_{n,0}=\Pi(A_n)$.  Then, as before, it is easy to check that $L_{n,i}/L_{n,i-1} = \fhat_{i-1}^{A_n}(Y_i) / \fbest(Y_i)$, $i=1,\ldots,n$.  Also 
\[ \E\{(L_{n,i}/L_{n,i-1})^{1/2} \mid \Y_{i-1}\} = 1-h^\star(\fbest,\fhat_{i-1}^{A_n}), \quad i=1,\ldots,n. \]
We can now anticipate a version of Lemma~\ref{lem:rates} for the mis-specified case.  

\begin{lem}
\label{lem:rates.miss}
Let $\Pi$ be $\eps_n$-thick at $\fbest$, with $C$ as in \eqref{eq:thick.miss}.  If $A_n$ is convex and $d\eps_n^2$-separated from $\fbest$, with $d > C+1$, then $\E(L_{n,n}^{1/2}) \leq \Pi(A_n)^{1/2} e^{-dn\eps_n^2}$. 
\end{lem}

\begin{proof}
Same as that of Lemma~\ref{lem:rates}.  
\end{proof}

To get a posterior convergence rate result, we must choose sets to be convex and suitably separated, with respect to $h^\star$, from $\fbest$.  For this, a natural choice would be $H^\star$-balls.  Indeed, the triangle inequality argument before, and the fact that ${H^\star}^2 \leq h^\star/2$ shows that $H^\star$-balls centered away from $\fbest$ with sufficiently small radius are separated from $\fbest$.  Technically, a more complicated notion of ``covering numbers for testing under mis-specficiation'' are needed in these cases.  However, if we assume $\FF$ is convex, for simplicity, then these special covering numbers are bounded by ordinary $H^\star$-covering numbers.  See \citet{kleijn}, Lemmas~2.1, 2.3, and the mixture model example in their Section~3. 

\begin{thm}
\label{thm:rates.miss}
Let $\FF$ be convex and $\Pi$ be $\eps_n$-thick at $\fbest$ with constant $C$ as in \eqref{eq:thick.miss}.  Suppose there exists a sequence $(\FF_n)$ such that $\Pi(\FF_n^c) \lesssim e^{-rn\eps_n^2}$, where $r > C+1$, and $\log N(\eps_n, \FF_n, H^\star) \lesssim n\eps_n^2$.  Then $\Pi_n(\{f: H^\star(\fbest,f) \gtrsim \eps_n\}) \to 0$ in probability.  
\end{thm}

\begin{proof}
Similar to that of Theorem~\ref{thm:rates}.  
\end{proof}


\section{Extension to independent non-iid models}
\label{S:independent}

\subsection{Setup and notation}

Let $Y_1,\ldots,Y_n$ be independent but not necessarily iid.  To formulate this, we shall use some slightly different notation compared to the previous sections.  Suppose that $Y_i \sim f_{\theta i}$, where, for each $\theta \in \Theta$, $f_{\theta i}$ is a density with respect to a measure $\mu_i$ on $\YY_i$, $i=1,\ldots,n$.  An important example is the fixed-design Gaussian regression, i.e., $Y_i \sim \nm(\theta(x_i), 1)$, where $x_i$ is a fixed covariate and $\theta(\cdot)$ is an unknown regression function.  The new $\theta$ notation is simply to indicate that there is a single unknown characteristic $\theta$, common to all $i=1,\ldots,n$; the manner in which $\theta$ is used can differ across $i$, however.  

Let $\theta^\star$ denote the ``true'' value of $\theta$.  As before, define the likelihood ratio as 
\[ R_n(\theta) = \prod_{i=1}^n f_{\theta i}(Y_i)/f_{\theta^\star i}(Y_i). \]
If $\Pi$ is a prior distribution on $\Theta$, then the posterior distribution for $\theta$, given observations $Y_1,\ldots,Y_n$, is given by 
\begin{equation}
\label{eq:post.indep}
\Pi_n(B) = \Pi(B \mid Y_1,\ldots,Y_n) = \frac{\int_B R_n(\theta) \,\Pi(d\theta)}{\int_\Theta R_n(\theta) \,\Pi(d\theta)}, \quad B \subseteq \Theta.
\end{equation}
The goal is to show that $\Pi_n(B_n^c) \to 0$ for $B_n$ a shrinking neighborhood of $\theta^\star$.  Next we restate our main definitions.  

\begin{description}
\item[\it Prior thickness.]  Let 
\[ \Kbar_n(\theta^\star,\theta) = \frac1n \sum_{i=1}^n K(f_{\theta^\star i}, f_{\theta i}) \quad \text{and} \quad \Vbar_n(\theta^\star,\theta) = \frac1n \sum_{i=1}^n V(f_{\theta^\star i}, f_{\theta i}), \]
where $K$ and $V$ are defined in Section~\ref{SS:thickness}.  Then we say the prior $\Pi$ is $\eps_n$-thick at $\theta^\star$ if 
for some constant $C > 0$, 
\begin{equation}
\label{eq:thick.indep}
\Pi(\{\theta: \Kbar_n(\theta^\star,\theta) \leq \eps_n^2,\, \Vbar_n(\theta^\star,\theta) \leq \eps_n^2\}) \geq e^{-Cn\eps_n^2}.
\end{equation}
It follows from Lemma~10 of \citet{ghosalvaart2007a} that the conclusion of Lemma~\ref{lem:thick} above holds in the independent non-iid case.  That is,
\begin{equation}
\label{eq:thick.lemma.indep}
\text{$\prob(I_n \leq e^{-cn\eps_n^2}) \to 0$ for any $c > C+1$}, 
\end{equation}
where $I_n = \int R_n(\theta) \,\Pi(d\theta)$ is the denominator in \eqref{eq:post.indep}.  

\item[\it Separation.] For a distance on $\Theta$, we shall employ a type of mean-Hellinger distance $H_n$, whose square is given by
\[ H_n(\theta^\star,\theta)^2 = \frac1n \sum_{i=1}^n H(f_{\theta^\star i}, f_{\theta i})^2. \]
As usual, set $h_n = H_n^2/2$.  Then we say that $\theta^\star$ and a set $A \subseteq \Theta$ are $\delta$-separated (with respect to $h_n$) if $h_n(\theta^\star,\theta) > \delta$ for all $\theta \in A$.  
\end{description}

\subsection{Convergence rate results}

Before stating the independent non-iid version of Proposition~\ref{prop:barron}, we need to set some more notation.  Let $\fhat_{(i-1)i}$ denote the predictive density of $Y_i$ given $Y_1,\ldots,Y_{i-1}$, i.e., 
\[ \fhat_{(i-1)i}(y) = \int f_{\theta i}(y) \,\Pi_{i-1}(d\theta), \quad i=1,\ldots,n. \]

\begin{prop}
\label{prop:barron.indep}
For a sequence $(\eps_n)$, with $\eps_n \to 0$, let $\KK_n=\{\theta: \Kbar_n(\theta^\star,\theta) \leq \eps_n^2\}$.  If $\log \Pi(\KK_n) \gtrsim -n\eps_n^2$, then $n^{-1} \sum_{i=1}^n \E\{K(f_{\theta^\star i},\fhat_{(i-1)i})\} \lesssim \eps_n^2$.  
\end{prop}

\begin{proof}
The proof is similar to that of Proposition~\ref{prop:barron} once we set the appropriate notation, etc.  Let $\fhat^n$ denote the joint density for $(Y_1,\ldots,Y_n)$ under the Bayes model.  Then, just like in the proof of Proposition~\ref{prop:barron}, 
\[ \fhat^n(Y_1,\ldots,Y_n) = \int \prod_{i=1}^n f_{\theta i}(Y_i) \, \Pi(d\theta) = \prod_{i=1}^n \fhat_{(i-1)i}(Y_i). \]
It follows that $K(f_{\theta^\star}^n, \fhat^n) = \sum_{i=1}^n \E\{K(f_{\theta^\star i}, \fhat_{(i-1)i})\}$.  Therefore, we can safely work with the notationally simpler $n^{-1}K(f_{\theta^\star}^n, \fhat^n)$.  From this point, follow the proof of Proposition~\ref{prop:barron}, i.e., restrict $\Theta$ to $\KK_n$ and use convexity of the Kullback--Leibler number.  
\end{proof}

For posterior convergence rates, take a sequence of subsets $(A_n)$ in $\Theta$ and let $L_{n,i}$ be the numerator of $\Pi_i(A_n)$ in \eqref{eq:post.indep}, where $L_{n,0}=\Pi(A_n)$.  As before, we have
\[ L_{n,i}/L_{n,i-1} = \fhat_{(i-1)i}^{A_n}(Y_i) / f_{\theta^\star i}(Y_i), \quad i=1,\ldots,n, \quad n \geq 1. \]
where $f_{(i-1)i}^{A_n}$ is the predictive density from before, but with the posterior $\Pi_{i-1}$ restricted to the set $A_n \subset \Theta$.  Also, 
\[ \E\{(L_{n,i}/L_{n,i-1})^{1/2} \mid \Y_{i-1}\} = 1-h(f_{\theta^\star i}, \fhat_{(i-1)i}^{A_n}), \]
where $h=H^2/2$ and $H$ is the usual Hellinger distance between densities.  With this, we are ready for an analogue of Lemma~\ref{lem:rates} for the independent non-iid case.  

\begin{lem}
\label{lem:rates.indep}
Let $\Pi$ be $\eps_n$-thick at $\theta^\star$, with $C$ as in \eqref{eq:thick.indep}.  If $A_n$ is convex and $d\eps_n^2$-separated from $\theta^\star$, with $d > C+1$, then $\E(L_{n,n}^{1/2}) \leq \Pi(A_n)^{1/2} e^{-dn\eps_n^2}$. 
\end{lem}

\begin{proof}
Just like the proof of Lemma~\ref{lem:rates}.  
\end{proof}

In the following theorem, we shall also need a type of max-Hellinger distance, 
\[ H_{n,\infty}(\theta^\star,\theta) = \max_{1 \leq i \leq n} H(f_{\theta^\star i}, f_{\theta i}). \]
Also let $h_{n,\infty} = H_{n,\infty}^2/2$.  This additional sort of distance will be needed for the general construction of sets which are both convex and sufficiently separated from $\theta^\star$.  Some remarks on removing the need for $H_{n,\infty}$ are given following the proof.  

\begin{thm}
\label{thm:rates.indep}
Let $\Pi$ be $\eps_n$-thick at $\theta^\star$ with constant $C$ as in \eqref{eq:thick.indep}.  Suppose there exists a sequence $(\Theta_n)$ such that $\Pi(\Theta_n^c) \lesssim e^{-rn\eps_n^2}$, where $r > C+1$, and $\log N(\eps_n, \Theta_n, H_{n,\infty}) \lesssim n\eps_n^2$.  Then $\Pi_n(\{\theta: H_n(\theta^\star, \theta) \gtrsim \eps_n\}) \to 0$ in probability.  
\end{thm}

\begin{proof}
For a constant $M > 0$ to be determined, let $B_n = \{\theta: H_n(\theta^\star,\theta) > M\eps_n\}$.  It suffices to show that $\Pi_n(B_n \cap \Theta_n) \to 0$ in probability.  We cover $B_n \cap \Theta$ by $H_{n,\infty}$-balls $A_{nj}$ of radius $M\eps_n/2$ with centers in $B_n$, where $j=1,\ldots,J_n$ and $J_n \leq e^{Rn\eps_n^2}$, $R > 0$.  That is, for suitable $\theta_j$ satisfying $H_n(\theta^\star,\theta_j) > M\eps_n$, take 
\[ A_{nj} = \{\theta: H_{n,\infty}(\theta_j, \theta) < M\eps_n/2\}, \quad j=1,\ldots,J_n. \]
Note the use of $H_{n,\infty}$ in the definition of $A_{nj}$ as opposed to $H_n$.  Everything will carry through as before as soon as we show that the $A_{nj}$ are convex and $(M^2\eps_n^2/8)$-separated from $\theta^\star$, with respect to $h_n$.  For convexity, let $\Phi_i$, $i=1,\ldots,n$, be any probability measure on $A_{nj}$.  By convexity of $h$ we get
\[ h(f_{\theta_j i}, f_{\Phi_i i}) \leq \int_{A_{nj}} h(f_{\theta^\star i}, f_{\theta i}) \, \Phi_i(d\theta), \quad i=1,\ldots,n. \]
By definition of $A_{nj}$, the right-hand side is bounded by $M^2\eps_n^2/8$.  From here it follows that $A_{nj}$ is convex and, in particular, predictive densities $\fhat_{(i-1)i}^{A_{nj}}$, restricted to $A_{nj}$, have properties like those densities $f_{\theta i}$ with $\theta \in A_{nj}$.  Our use of the max-Hellinger metric is necessary here because the measures $\Phi_i$ can vary with $i$, just like the posteriors $\Pi_i^{A_{nj}}$ vary with $i$.  Now, for separation, given $\theta \in A_{nj}$, the triangle inequality for $H_n$ gives 
\[ H_n(\theta^\star, \theta) \geq H_n(\theta^\star, \theta_j) - H_n(\theta_j, \theta). \]
The first term is greater than $M\eps_n$ by the choice of $\theta_j$.  The second term is less than $H_{n,\infty}(\theta_j,\theta)$ which is less than $M\eps_n/2$ by the definition of $A_{nj}$.  Therefore, $\theta^\star$ and $A_{nj}$ are $(M^2\eps_n^2/8)$-separated with respect to $h_n$.  Now we may apply Lemma~\ref{lem:rates.indep} to each $A_{nj}$ just like in the proof of Theorem~\ref{thm:rates}, to show that if $M$ is large enough, then $\Pi_n(\{\theta: H_n(\theta^\star,\theta) > M\eps_n\}) \to 0$ in probability.  
\end{proof}

The use of a max-Hellinger metric in Theorem~\ref{thm:rates.indep} can be avoided in some cases, e.g., if $H_n$ is equivalent to some fixed metric on $\theta$-space.  One specific example is nonparametric regression using splines; see \citet[][Sec.~7.7]{ghosalvaart2007a}.

\section{Extension to Markov process models}
\label{S:markov}

\subsection{Setup and notation}
\label{SS:setup.markov}

Let $(Y_n: n \geq 0)$ be an ergodic Markov process on $\YY$ with transition density $f_\theta(y' \mid y)$ and stationary density $u_\theta(y)$, both with respect to a $\sigma$-finite measure $\mu$ on $\YY$, and both indexed by $\theta \in \Theta$.  That is, the transition density $f_\theta$ characterizes the one-step moves $Y_n \to Y_{n+1}$ of the process, and $u_\theta$ the limiting marginal distribution of $Y_n$.  Here, like in \citet{ghosalvaart2007a}, we assume the process is at stationarity, i.e., that $Y_0 \sim u_{\theta^\star}$, so that all the marginal distributions are the same and equal to $u_{\theta^\star}$.  The goal here is estimation of the unknown index $\theta$.  Methods developed in the previous sections, particularly in Section~\ref{S:independent}, shall be used to prove posterior convergence rate theorems.  

Following the previous setup, let $\theta^\star$ denote the ``true'' $\theta$ value.  Now define 
\[ R_n(\theta) = \prod_{i=1}^n \frac{f_\theta(Y_i \mid Y_{i-1})}{f_{\theta^\star}(Y_i \mid Y_{i-1})} \cdot \frac{u_\theta(Y_0)}{u_{\theta^\star}(Y_0)} \]
as the likelihood ratio for $(Y_0,Y_1,\ldots,Y_n)$.  For a prior distribution $\Pi$ on $\Theta$ and a measurable set $A \subset \Theta$, Bayes theorem gives the posterior distribution for $f$ as follows:
\begin{equation}
\label{eq:post.markov}
\Pi_n(A) = \Pi(A \mid Y_0,\ldots,Y_n) = \frac{\int_A R_n(\theta) \, \Pi(d\theta)}{\int_\Theta R_n(\theta) \, \Pi(d\theta)}. 
\end{equation}
The primary goal of this section is to investigate the convergence of $\Pi_n(A_n)$, where $A_n$ is the complement of a shrinking neighborhood of $\theta^\star$.  The notion of ``neighborhood'' is more difficult here than in the previous cases; see Section~\ref{SS:main.markov} below.  

\begin{description}
\item[\it Prior thickness.] For concentration properties of the prior $\Pi$, consider  
\begin{align*}
K(\theta^\star,\theta) & = \int K_y(f_{\theta^\star}, f_\theta) u_{\theta^\star}(y) \,\mu(dy), \\
V(\theta^\star,\theta) & = \int V_y(f_{\theta^\star}, f_\theta) u_{\theta^\star}(y) \,\mu(dy),  
\end{align*}
where $K_y(f_{\theta^\star},f_\theta) = K(f_{\theta^\star}(\cdot \mid y), f_\theta(\cdot \mid y))$ is the usual Kullback--Leibler divergence for densities; $V_y$ is defined similarly, for $V$ as in Section~\ref{SS:thickness}.  Let $\Theta_0$ be the set of all $\theta$'s such that both $K(u_{\theta^\star}, u_\theta)$ and $V(u_{\theta^\star}, u_\theta)$ are bounded by~1.  With this notation, we say that the prior $\Pi$ is $\eps_n$-thick at $\theta^\star$ if, for some constant $C > 0$, 
\begin{equation}
\label{eq:thick.markov}
\Pi(\{\theta \in \Theta_0: K(\theta^\star,\theta) \leq \eps_n^2, \, V(\theta^\star,\theta) \leq \eps_n^2\}) \geq e^{-Cn\eps_n^2}. 
\end{equation}
Lemma~10 of \citet{ghosalvaart2007a} gives an analogue of Lemma~\ref{lem:thick} above for the present dependent data case.  That is, for $C$ as in \eqref{eq:thick.markov}, 
\begin{equation}
\label{eq:thick.lemma.markov}
\prob(I_n \leq e^{-cn\eps_n^2}) \to 0 \quad \text{for any $c > C+1$},
\end{equation}
where $I_n$ is the denominator in \eqref{eq:post.markov}.  

\item[\it Separation.] Let $H_y$ be the usual Hellinger distance on transition densities with fixed state $y$, i.e., $H_y(f_{\theta^\star}, f_\theta) = H(f_{\theta^\star}(\cdot \mid y), f_\theta(\cdot \mid y))$.  Also let $h_y=H_y^2/2$.  Now define the max-Hellinger (semi)metric $H_\infty(\theta^\star,\theta) = \sup_y H_y(\theta^\star,\theta)$.  We say that $\theta^\star$ and a set $A \subseteq \Theta$ are $\delta$-separated (with respect to $h_\infty$) if $h_\infty(\theta^\star,\theta) > \delta$ for all $\theta \in A$.  
\end{description}

\subsection{Convergence rate results}
\label{SS:main.markov}

To start, consider the predictive density problem of Section~\ref{S:predictive}.  In this case, the predictive density is itself a transition density.  In particular, we have 
\[ \fhat_{i-1}(y \mid Y_{i-1}) = \int f_\theta(y \mid Y_{i-1}) \,\Pi_{i-1}(d\theta), \quad i=1,\ldots,n, \]
the expected transition density with respect to the posterior distribution $\Pi_{i-1}$.  This is a typical Bayes estimate of the transition density, and the claim is that it converges to the true transition density $f_{\theta^\star}$ as $n \to \infty$.  In particular, we have the following convergence rate result for predictive densities.  

\begin{prop}
\label{prop:barron.markov}
Given $(\eps_n)$ with $\eps_n \to 0$ and $n\eps_n^2 \to \infty$, let $\KK_n = \{\theta: K(\theta^\star,\theta) \leq \eps_n^2, \,K(u_{\theta^\star},u_\theta) < \infty\}$.  If $\log \Pi(\KK_n) \gtrsim -n\eps_n^2$.  Then $n^{-1}\sum_{i=1}^n \E\{K_{Y_{i-1}}(f_{\theta^\star}, \fhat_{i-1})\} \lesssim \eps_n^2$.  
\end{prop}

\begin{proof}
Let $\fhat^n$ denote the joint density for $(Y_0,\ldots,Y_n)$ under the Bayes model.  Then 
\[ \fhat^n(Y_0,\ldots,Y_n) = \int u_\theta(Y_0) \prod_{i=1}^n f_\theta(Y_i \mid Y_{i-1}) \, \Pi(d\theta) = \hat u_0(Y_0) \prod_{i=1}^n \fhat_{i-1}(Y_i \mid Y_{i-1}), \]
just like in the proof of Proposition~\ref{prop:barron.indep}, where $\hat u_0(Y_0) = \int u_\theta(Y_0) \,\Pi(d\theta)$.  Another simple calculation shows that if $f_{\theta^\star}^n$ is the joint distribution of $(Y_0,Y_1,\ldots,Y_n)$ under $\theta^\star$, then the joint Kullback--Leibler divergence $K(f_{\theta^\star}^n, \fhat^n)$ equals 
\[ \E \Bigl\{ \log \frac{u_{\theta^\star}(Y_0) \prod_{i=1}^n f_{\theta^\star}(Y_i \mid Y_{i-1})}{\hat u_0(Y_0) \prod_{i=1}^n \fhat_{i-1}(Y_i \mid Y_{i-1})} \Bigr\} = \sum_{i=1}^n \E\{ K_{Y_{i-1}}(\ftrue, \fhat_{i-1})\} + K(u_{\theta^\star},\hat u_0). \]
Observe that $n^{-1}K(u_{\theta^\star},\hat u_0) = O(\eps_n^2)$ and, on $\KK_n$, $n^{-1}K(u_{\theta^\star},u_\theta) = O(\eps_n^2)$.  From here, the proof is just like that of Proposition~\ref{prop:barron}.  
\end{proof}

As before, the assumption of Proposition~\ref{prop:barron.markov} is implied by prior thickness at $\theta^\star$.  The theorem also extends the result in Corollary~2.1 of \citet{ghosal.tang.2006}. Indeed, our result is $n^{-1}\sum_{i=1}^n K_{Y_{i-1}}(\ftrue, \fhat_{i-1}) = O_P(\eps_n^2)$, which is stronger than the $o_P(1)$ obtained by these authors.  Our version of the Kullback--Leibler property is more strict than theirs, but this is typical when convergence rates are sought.  

Let $(A_n)$ be a sequence of measurable subsets of $\Theta$, and let $L_{n,i} = \int_{A_n} R_i(\theta) \,\Pi(d\theta)$ be the numerator of the posterior probability $\Pi_i(A_n)$ in \eqref{eq:post.markov}.  Then 
\[ L_{n,i} / L_{n,i-1} = \fhat_{i-1}^{A_n}(Y_i \mid Y_{i-1}) / f_{\theta^\star}(Y_i \mid Y_{i-1}), \quad i=1,\ldots,n, \quad n \geq 1, \]
where $\fhat_{i-1}^{A_n}$ is the predictive transition density when the posterior $\Pi_{i-1}$ is restricted to $A_n$.  We also have
\begin{equation}
\label{eq:h.y}
\E\{(L_{n,i} / L_{n,i-1})^{1/2} \mid \Y_{i-1}\} = 1-h_{Y_{i-1}}(f_{\theta^\star}, \fhat_{i-1}^{A_n}), \quad i=1,\ldots,n. 
\end{equation}
We can now present an extension of Lemma~\ref{lem:rates} for the case of Markov processes.   

\begin{lem}
\label{lem:rates.markov}
Let $\Pi$ be $\eps_n$-thick at $\theta^\star$, with $C$ as in \eqref{eq:thick.markov}.  If $A_n$ is convex and $d\eps_n^2$-separated from $\theta^\star$, with $d > C+1$, then $\E(L_{n,n}^{1/2}) \leq \Pi(A_n)^{1/2} e^{-dn\eps_n^2}$.   
\end{lem}

\begin{proof}
Exactly the same as that of Lemma~\ref{lem:rates}.  
\end{proof}

The fact that data $Y_{i-1}$ appears as part of the formula for the distance $h_{Y_{i-1}}$ in \eqref{eq:h.y} necessitates the use of the max-Hellinger metric, i.e., separation with respect to $h_\infty$ implies separation with respect to $h_y$ for any $y$, even if $y$ is random.  But we are free to formulate the convergence rate theorem with a different metric.  Here we consider $H_Q(\theta^\star,\theta) = \int H_y(f_{\theta^\star}, f_\theta) \,Q(dy)$, where $Q$ is a probability measure on $\YY$.  In the non-linear Gaussian autoregression example in \citet[][Sec.7.4]{ghosalvaart2007a}, the measure $Q$ is taken to be a two-point location mixture of Gaussians.  

\begin{thm}
\label{thm:rates.markov}
Let $\Pi$ be $\eps_n$-thick at $\theta^\star$ with constant $C$ as in \eqref{eq:thick.markov}.  Suppose there exists a sequence $(\Theta_n)$ such that $\Pi(\Theta_n^c) \lesssim e^{-rn\eps_n^2}$, where $r > C+1$, and $\log N(\eps_n, \Theta_n, H_\infty) \lesssim n\eps_n^2$.  Then $\Pi_n(\{\theta: H_Q(\theta^\star, \theta) \gtrsim \eps_n\}) \to 0$ in probability.  
\end{thm}

\begin{proof}
The proof is similar to that for the independent non-iid case.  In particular, we cover the complement of a mean-Hellinger neighborhood of $\theta^\star$ by max-Hellinger-balls.  The convexity and separation calculations are analogous to those in Theorem~\ref{thm:rates.indep}, and the rest of the argument goes just like in the proof of Theorem~\ref{thm:rates}.  
\end{proof}

\section{Discussion}
\label{S:discuss}

Here we have presented an analysis of Bayesian asymptotics based primarily on predictive densities.  These densities are fundamental quantities in Bayesian statistics, for they are Bayes density estimates under a variety of loss functions.  So, in this sense, our analysis has a stronger Bayesian flavor than other existing approaches.  We have also demonstrated how our basic approach can be tuned to handle a variety of models---iid, mis-specified iid, independent non-iid, and dependent Markov processes.  For example, essentially the same predictive density convergence rate result holds in all these contexts.  

We have opted here for simplicity of presentation rather than strength of results.  For example, one can easily tailor the analysis, taking more efficient choice of coverings, etc, to achieve sharper rates.  In particular, to achieve $n^{-1/2}$ rates in finite-dimensional parametric models, a special type of covering is required \citep[e.g.,][Theorem~2.4]{ggv2000}, and this can be incorporated into the present analysis.  On the other hand, if convergence of predictive densities is the only concern, then these special coverings are not necessary---only local thickness of the prior is needed.   Indeed, it is straightforward to follow the argument in \citet[][Sec.~7.7]{ghosalvaart2007a} to show that, in a nonparametric regression context, where the true regression function $\theta^\star$ lies in an $\alpha$-smooth function class, and a spline-based prior is used, the predictive densities converge, in the sense of Proposition~\ref{prop:barron.indep}, at the minimax rate $n^{-\alpha/(2\alpha+1)}$.  In this example, however, Ghosal and van der Vaart's analysis gives full convergence of the posterior at the same rate under basically the same assumptions.  But there may be some cases where the weaker conditions of the predictive density convergence theorems may be more useful.  
\begin{itemize}
\item Consider a basic iid Bayesian density estimation problem.  For Dirichlet process location-mixtures of Gaussians, care must be taken in choosing a prior for the common component scale $\sigma$.  This is like the choice of bandwidth in classical density estimation.  Typical conditions restrict the amount of mass the prior for $\sigma$ can place near zero.  However, these conditions are primarily needed for the control of model entropy---when $\sigma$ is near zero, the class of possible models is enormous, making the entropy large.  But if convergence of predictive densities is the question of interest, so that only local thickness is required, as in Proposition~\ref{prop:barron}, then entropy is not a concern.  Therefore, one can expect practically weaker assumptions on the prior if the focus is on convergence of predictive densities.  
\vspace{-2mm}
\item For dependent data models, there may be some advantage to the predictive density-based approach.  Indeed, in Section~\ref{S:markov}, convergence of the predictive densities in Proposition~\ref{prop:barron.markov} follows without any assumptions on the mixing of the process.  This is due to the fact that only ``first moment'' conditions---bounds on the Kullback--Leibler number---are needed.  Compare this to the posterior convergence analysis in \citet[][Sec.~4]{ghosalvaart2007a} which requires assumptions on the mixing of the process and some ``higher-than-second moment'' conditions.  
\end{itemize}

Finally, we mention that this investigation began by looking at a predictive density analysis of the posterior by using a law of large numbers for martingale difference arrays.  Unfortunately, that approach seemed to be somewhat limited; specifically, a non-trivial extension to a \emph{uniform} martingale law of large numbers is needed.  That idea is still interesting, see \citet{martin.hong.mlln}, although the results here are stronger.



\bibliographystyle{/Users/rgmartin/Research/TexStuff/asa}
\bibliography{/Users/rgmartin/Research/mybib}

\end{document}